\documentclass[12 pt]{amsart}

\newtheorem{theorem}{Theorem}[section]
\newtheorem{lemma}[theorem]{Lemma}
\newtheorem{proposition}[theorem]{Proposition}

\theoremstyle{definition}

\newtheorem{example}[theorem]{Example}
\newtheorem{remark}[theorem]{Remark}

\theoremstyle{remark}

\usepackage{amscd,amssymb}

\title{Degree estimate for subalgebras}
\author{yun-chang li and jie-tai yu}
\address{Department of Mathematics,
The University of Hong Kong, Hong Kong SAR, China} \email{liyunch@hku.hk,\ liyunch@163.com}
\address{Department of Mathematics, The University of Hong
Kong, Hong Kong SAR, China} \email{yujt@hkusua.hku.hk,\
yujietai@yahoo.com}
\thanks{The research of Jie-Tai Yu was partially
supported by an RGC-CERG Grant.}
\thanks{The research of Yun-Chang Li was partially supported by
a Postgraduate Studentship.}

\begin{document}

\begin{abstract}
Based on Bergman's Lemma on centralizers, we obtain a sharp lower degree bound for nonconstant elements in a
subalgebra generated by two elements of a free associative algebra over an arbitrary field.
\end{abstract}

\bigskip

\maketitle

\section{Introduction and the main result}

\noindent Let   $A_n=K\langle x_1,\cdots,x_n\rangle$  be the  free associative algebra of rank $n$ over a field $K$,\  $B$ a subalgebra of $A_n$ generated by two elements in $A_n\backslash K$.

\

\noindent Based on Bergman's Lemma on radicals \cite{B1}
that if the leading monomial of an element
in a Malcev-Neumann (power series) algebra (\cite{Ma, Ne1, Ne2, Ha}) over a field of characteristic $0$
has $n^{-th}$ roots, then so does the element itself,  Makar-Limanov and Yu \cite{MLY} gave a sharp lower degree bound for nonconstant elements in $B$ when the characteristic of $K$ is zero.

\

\noindent However, in the case of positive characteristic, the Lemma on radical is not true,
which can be shown by the following simple example that $x^2+x$ has
no square roots in the Malcev-Neumann (power series) algebra $F((x_1,\cdots,x_n))$ in free case over a field
$F$ of characteristic $2$. Therefore,
the method in \cite {MLY} is no longer applicable.

\

\noindent In this paper, based on  Bergman's Lemma on centralizers \cite{B1}, we generalize the degree estimate
in \cite{MLY} for any characteristic.

\begin{theorem}
Let $A_n=K\langle x_1,\cdots,x_k\rangle$ be a free associative algebra over a field $K$  and let $f, g\in A_n$ be algebraically independent elements over $F$. Suppose  the leading monomials $v(f)$ and $v(g)$ are algebraically dependent over $K$, and, neither $\deg(f)$ divides $\deg(g)$ nor $\deg(g)$ divides $\deg(f)$. Then for any
$P(x,y)\in K\langle x,y\rangle\backslash K$,
$$\deg(P(f,g))\geq w_{\deg(f), \deg(g)}(P(x,y))\frac{\deg([f,g])}{\deg(f)+\deg(g)}.$$
\end{theorem}

\

\section{Proof of the main result}

\noindent First we introduce some terminologies.
Let $K$ be a field of characteristic $r$(zero or prime), $A_n$  the free associative algebra generated by $X=\{x_1,\cdots,x_n\}$ over $K$ where $n\geq 2$, and $F=\langle X\rangle$ be the free group generated by $X$. By a group order, we mean that it is a total order of the group as a set, and coincides to the operation of the group as well; namely, if a group $G$ has a group order, then $G$ is totally ordered as a set, and to any $a,b,c\in G$, if $a>b$, we always have $ca>cb$ and $ac>bc$.  Since it is possible to equip $F$ a group order which is an extension of the partial order of the total degree \cite{Ne2}, namely if $\deg(a(x_1,\cdots,x_n))>\deg(b(x_1,\cdots,x_n))$ where $a(x_1,\cdots,x_n),b(x_1,\cdots,x_n)\in F$, then $a(x_1,\cdots,x_n)>b(x_1,\cdots,x_n)$, $K((F))$ forms a Malcev-Neumann algebra \cite {Ma, Ne1, Ha} under this order. Any element $f\in A_n$ can be viewed as an element of $K((F))$. Let the leading term (namely the least element in the support) of $f$ be $c\cdot h$ with $c\in K^*$ and $h\in F$, we denote $h$ by $v(f)$ and $c$ by $c(f)$. For the degree functions, let $\deg$ be the total degree, or homogeneous degree, of a polynomial in $K((F))$ and $\deg_{x_i}$ be the partial degree relative to $x_i$. Here we will restate the definition of weighted degree of a polynomial which has been defined in \cite{B1, B2} just for convenience. The weighted degree $w_{k_1,\cdots,k_n}(m(x_1,
\cdots,x_n))$ of a monomial $m$ is equal to $\sum_{i=1}^nk_i\cdot\deg_{x_i}(m)$, and for a polynomial $p(x_1,\cdots,x_n)$, $w_{k_1,\cdots,k_n}(p)=max\{w_{k_1,\cdots,k_n}(m)|m\in supp(p)\}$. Obviously we have $\deg(m)=w_{1,\cdots,1}(m)$ and $\deg_{x_i}(m)=w_{0,\cdots,0,1,0,\cdots,0}$ where 1 is the $i^{-th}$ coordinite.

Let $f,g\in A_n$ be algebraically independent where $v(f)$ and $v(g)$ are algbraically dependent but $\deg(f)\nmid\deg(g)$, $\deg(g)\nmid\deg(f)$, and we assume that $\deg(g)=n>m=\deg(f)$.

\

\noindent Crucial to the proof of Theorem 1.1 is the following Bergman's Lemma on Centralizers \cite{B1, B2}.

\

\begin{lemma}[{\bf on centralizers}]
Let $R$ be a commutative ring, $S$ an ordered semigroup (the group order), and an element of $R((S))$ with invertible leading term $a_uu$. (Thus, $u$ is invertible in $S$, and $a_u$ in $R$). Then there exists an element $f$ with leading term $1$, such that the element $c=f^{-1}af$ (which clearly also has leading term $a_uu$) has support entirely in the centralizer of $u$ in $S$.
\end{lemma}

\noindent Now we re-present the proof of Lemma on centralizers  in \cite {B1, B2} for self-contain-ness of this paper
as the journal that \cite{B1, B2} appeared is not well circulated.

\begin{proof}

\noindent Clearly, we may assume without loss of generality that $a_u=1$.

\noindent Let $\infty$ be a symbol outside of $S$ with the property $\forall s\in S$, $s<\infty$, and let $S'=S\cup \{\infty\}$. Of course $S'$ is a totally ordered set. By `the leading term of $r\in R((S))$ is $\alpha t$', we mean that if $r=0$, then $t=\infty$ and $\alpha$ is undefined. To each pair $x,y\in S'$, the intervals of different types are defined as follows:  $[x,y]=\{s\in S'|x\leq s\leq y\}$; $(x,y]=\{s\in S'|x<s\leq y\}$; $[x,y)=\{x\in S'|x\leq s<y\}$; $(x,y)=\{s\in S'|x<s<y\}$.

\

\noindent For $s,t\in S$, $s$ being invertible, we define $\frac{t}{s}=max\{ts^{-1},s^{-1}t\}$. We also define $\frac{\infty}{s}=\infty$. Easy to get that $x>y$ implies $\frac{x}{s}>\frac{y}{s}$.

\

\noindent Let $X$ be the set of all 3-tuples $(t,b,e)$ where $t\in (u,\infty]$, $b\in R((S))$ with $v(b)=u, c(b)=1$ and $supp(b)\subseteq [u,t)\cap C_u(S)$, and $e$ is an element with leading term 1 and support in $[1, \frac{t}{u})$ such that $v(ebe^{-1}-a)=t, c(ebe^{-1}-a)=\alpha$(here we mean that if $ebe^{-1}-a=0$, then $t=\infty$, and if not, $\alpha\in R-\{0\}$).

\

\noindent Now establish a partial order on $X$: $(t,b,e)<(t',b',e')$ if and only if $t<t'$, $\text{supp}(b'-b)\subseteq [t,t')$ and $\text{supp}(e'-e)\subseteq [\frac{t}{u}, \frac{t'}{u})$ (here notice that surely $\frac{t}{u}<\frac{t'}{u}$ as being proved). The last two conditions say that $b'$, $e'$ "extend" $b$ and $e$.

\

\noindent $X$ is nonempty since $(v(a-u),u,1)\in X$. Hence, to each ascending chain $\{(t_l,b_l,e_l)|l\in N^+\}$, we just `piece together' $b_l$ and $e_l$ as $b$ and $e$, and let $t=v(ebe^{-1}-a)$ (obviously here $t\geq t_l$ for each $l$), and then $(t,b,e)$ becomes the upper bound of the chain. Hence, according to Zorn's Lemma, $X$ has a maximal one.

\

\noindent We now prove that if $t<\infty$, $(t,b,e)$ can not be a maximal element.  If not, let $(t,b,e)$ with $t<\infty$ be a maximal element, and we have three cases.

\

\noindent ${\bf Case\ 1}$. $tu^{-1}>u^{-1}t$. Then $\frac{t}{u}=tu^{-1}$. Let $e'=e-\alpha tu^{-1}$, and hence $e'^{-1}=e^{-1}+\alpha tu^{-1}+o(tu^{-1})$ where $o(tu^{-1})$ means that it is an element of $R((S))$ each of whose support is greater that $tu^{-1}$. Let $b'=b$, and $t'=v(e'b'e'^{-1}-a)$. Since $(e-\alpha tu^{-1})b(e^{-1}+\alpha tu^{-1}+o(tu^{-1}))-a=(ebe^{-1}-a)+\alpha ebtu^{-1}+ebo(tu^{-1})-\alpha tu^{-1}be^{-1}-\alpha^2tu^{-1}btu^{-1}-\alpha tu^{-1}bo(tu^{-1})$, and $v(\alpha ebtu^{-1})=utu^{-1}>t$, $v(ebo(tu^{-1}))>utu^{-1}>t$, $v(-\alpha^2 tu^{-1}btu^{-1})=tu^{-1}utu^{-1}=t^2u^{-1}>t$(notice that $t>u$), $v(-\alpha tu^{-1}bo(tu^{-1}))>tu^{-1}utu^{-1}>t$, $v((ebe^{-1}-a)-\alpha tu^{-1}be^{-1})=v((\alpha t+o(t))-\alpha tu^{-1}u+o(t))>t$, as well as $v((ebe^{-1}-a)+\alpha ebtu^{-1}+ebo(tu^{-1})-\alpha tu^{-1}be^{-1}-\alpha^2tu^{-1}btu^{-1}-\alpha tu^{-1}bo(tu^{-1}))\geq max\{v((ebe^{-1}-a)-\alpha tu^{-1}be^{-1}),$

\noindent $v(\alpha ebtu^{-1}), v(ebo(tu^{-1})), v(-\alpha^2tu^{-1}btu^{-1}), v(-\alpha tu^{-1}bo(tu^{-1}))\}$, $t'>t$. It means that $(t',b',e')>(t,b,e)$ which contradicts to $(t,b,e)$ being maximal.

\

\noindent ${\bf Case\ 2}$. $tu^{-1}<u^{-1}t$. Similar to case 1, we just let $e'=e-\alpha u^{-1}t, b'=b$,  and $v(e'b'e'^{-1}-a)>t$.

\

\noindent ${\bf Case\ 3}$. $tu^{-1}=u^{-1}t$. Then $t$ commutes with $u$, so we can let $e'=e, b'=b-\alpha t$, and hence $e'b'e'^{-1}-a=e(b-\alpha t)e^{-1}-a=(ebe^{-1}-a)-\alpha ete^{-1}$. Since $ebe^{-1}-a=\alpha t+o(t), v(\alpha ete^{-1})=t$, $v((ebe^{-1}-a)-\alpha ete^{-1})>t$, namely $t'>t$ which contradicts to $(t,b,e)$ being maximal.

\noindent Therefore, there must exist some $(t,b,e)$ such that $t=\infty$, namely $ebe^{-1}=a$, or $e^{-1}ae=b$.
\end{proof}

\noindent Let us give an example in $K((F))$ to understand Bergman's Lemma on centralizers and its proof. Here we will use the opposite definition of "well-ordered" on $F$, namely each subset has a greatest element.

\begin{example}
In $F$ we assume $x>y$ and $xy\cdot (x^2)^{-1}< (x^2)^{-1}\cdot xy$ (of course $xy\cdot (x^2)^{-1}>(x^2)^{-1}\cdot xy$ is also feasible since they are both extended total orders of the partial order of degree) and let $a=x^2+xy$. By Bergman's method, we establish the approximation starting from $(xy,x^2,1)(b=v(a), t=v(a-v(a)),e=1)$. Then $e'=e+xy\cdot (x^2)^{-1}=1+xy\cdot (x^2)^{-1}$ and $(e')^{-1}=e^{-1}-xy\cdot (x^2)^{-1}+O(xy\cdot (x^2)^{-1})=1-xy\cdot (x^2)^{-1}+O(xy\cdot (x^2)^{-1})$ where $O(xy\cdot (x^2)^{-1})$ means all the monomials behind are all less than $xy\cdot (x^2)^{-1}$. $b'=b=x^2$, and since $e'b'(e')^{-1}=(1+xy\cdot (x^2)^{-1})x^2(1-xy\cdot (x^2)^{-1}+O(xy\cdot (x^2)^{-1}))$, it is easy to get that $v(e'b'(e')^{-1}-a)=x^2\cdot xy\cdot (x^2)^{-1}$ since $x>y$, namely $t'=x^2\cdot xy\cdot (x^2)^{-1}$.

\

\noindent After $k$ steps, we get the three-tuple  $(t_k,b_k,e_k)$. Now we claim that to $all$ the $t_i's$, if $t_i\not=\infty$, then $deg(t_i)=2$, and all the $e_i's$ are homogenous of degree 0 and $b_i=x^2$ all the way. For $k=1$, we see $t_1=x^2\cdot xy\cdot (x^2)^{-1}$, $e_1=1+xy\cdot (x^2)^{-1}$, $b_1=x^2$ and it satisfies. Assume that it is correct for $k=n-1$. If $t_{n-1}=\infty$, then $e_{n-1}b_{n-1}e^{-1}_{n-1}=a$, and we prove it. If not, since $t_{n-1}$ is a monomial of degree 2 however it is less than $x^2$, so it can not commute with $x^2$\ (By Bergman \cite{B},  the centralizer of any element of $K\langle x_1,\dots,x_n\rangle\backslash K$ is a polynomial algebra in one variable over $K$).  Hence $b_n=b_{n-1}=x^2$, $e_n=e_{n-1}+\alpha t_n\cdot x^{-2}/\alpha x^{-2}\cdot t_n$,\ and the new term of $e_n$ will always has degree $0$. Then $e_n$ is also homogenous of degree $0$ and so is $e_n^{-1}$. Obviously $e_nb_ne_n^{-1}$ is homogenous of degree $2$ and since $a$ is homogenous of degree $2$, $e_nb_ne_n^{-1}-a$ is homogeneous of degree $2$ or equal to $0$, namely $\deg(t_n)=2$ or $t_n=\infty$.

\

\noindent It means that after finite steps of the algorithm, we always get $eae^{-1}=x^2+t$ where $\deg(t)=2$, or we get $eae^{-1}=x^2$. Now we consider the subset $S$ of three-tuples $(t,b,e)$ defined in the proof of Lemma on centralizers where $e$ being homogeneous of degree 0 and $b=x^2$. Since $a$ is homogenous of degree $2$,\ $t$ is also of degree 2 or $\infty$. Then, By preserving the order introduced by Bergman on $S$, if $t$ is not $\infty$, we can always construct an `extension' of $b$ and $e$ such that $(t',b',e')\in S$ is greater. However, by the `piece together', we will always get a maximal element, and hence we get the maximal element with $t=\infty$, namely there exists an $e$ which is homogenous of degree 0 such that $eae^{-1}=x^2$.
\end{example}

\begin{remark}
The steps in the proof of Bergman is to construct a `better approximative' element to the
maximal element instead of  calculating the maximal three-tuple.
\end{remark}

\noindent According to the discussion in the example above, we obtain

\begin{proposition}
If an element $a\in K((F))$ is homogenous, then there exists some $e\in K((F))$ with leading term $1$ which is homogenous of degree $0$ such that $eae^{-1}=c(a)v(a)$.
\end{proposition}\qed

\noindent Then according to Lemma on centralizers, there exists some $t\in K((F))$ with $c(t)v(t)=1$ such that  the support of $tft^{-1}$ is in $C_F(v(f))$. Let $v(f)=h^q$ where $h$ is the generator of $C_F(h)$, and then $tft^{-1}=\sum_{i=-\infty}^q a_ih^i$ with $a_i\in K$. Let $f'=tft^{-1}$, $g'=tgt^{-1}$, and we have the following

\begin{lemma}
For any $P(x,y)\in K\langle x,y\rangle$,\ $P(f',g')=tP(f,g)t^{-1}$.
\end{lemma}

\begin{proof}
Let $P(x,y)=\sum_{i=0}^m\sum_{j=0}^na_{ij}x^iy^j$ for some nonnegative integers $i$ and $j$ where $a_{ij}\in K$. Then $$P(f',g')=\sum_{i=0}^m\sum_{j=0}^na_{ij}f'^ig'^j=\sum_{i=0}^m\sum_{j=0}^na_{ij}(tft^{-1})^i(tgt^{-1})^j$$
$=\sum_{i=0}^m\sum_{j=0}^na_{ij}tf^ig^jt^{-1}=t(\sum_{i=0}^m\sum_{j=0}^na_{ij}f^ig^j)t^{-1}=tP(f,g)t^{-1}$.
\end{proof}

\

\noindent Since $v(t)=1$, $\deg(P(f,g))=\deg(P(f',g'))$ where the degree function is the homogenous degree of $K((F))$. So we can just do  degree estimate for $P(f',g')$.

\

\noindent Two elements of $A_n$ are called algebraically independent over $K$ if  they generate a subalgebra of rank two. If $v(f)$ and $v(g)$ are algebraically independent, then for all $P(x,y)\in K\langle x,y\rangle\backslash K$,\

\noindent $\deg(P(f,g))=w_{\deg(f),\deg(g)}(P(x,y))$, so we may assume without loss of generality that
$v(f)$ and $v(g)$ are algebraically dependent. However, if $\deg(f)\mid\deg(g)$ or $\deg(g)\mid\deg(f)$, then $\deg(f)+\deg(g)$ can be reduced by some automorphism, so we also assume $\deg(f)\nmid\deg(g)$ as well as $\deg(g)\nmid\deg(f)$. We  assume that $f$ and $g$ are algebraically independent over $K$ but $v(f)$ and $v(g)$ are not. Hence since $v(f')=v(f)$ and $v(g')=v(g)$, $f'$ and $g'$ are algebraically independent but $v(f')$ and $v(g')$ are algebraically dependent.  Then since $h$ generates its own centralizer in $A_n$,\ $v(g')=h^p$ for some positive integer $p$. Let $g'=h^p+g'_1$ where $v(g'_1)< h^p$, and if $v(g'_1)$ and $h$ are dependent, then $v(g'_1)=h^{p_1}$ for some integer $p_1$ which is less than $p$. This can be done inductively.

\begin{lemma}[\textbf{on\ steps}]

The above process  will stop after a finite number of steps.
\end{lemma}

\begin{proof}
After $k$ steps, let $g'=\sum_{i=1}^k a_ih^{m_i}+g'_k$. Obviously $\deg([f',g'])=\deg([f',g'_k])\leq deg(f')+deg(g'_k)$, so $\deg(g'_k)\geq \deg([f',g'])-\deg(f')=\deg([f,g])-\deg(f)$. Here notice that $\deg(h)>0$, so after each step, if possible, the $\deg(g'_i)$ decreases by at least 1 which means after at most
$$\deg(g)-(deg([f,g])-deg(f))=\deg(fg)-\deg([f,g])$$ steps, the process will stop.
\end{proof}

\noindent Hence, after a finite number of steps we  get $g'=\sum_{i=p-k}^p a_ih^{i}+s$ where $v(s)$ and $h$ are algebraically independent.

\

\noindent Let $C$ be the subalgebra generated by $h$,\ $h^{-1}$ and $s$, and equip it with the weighted degree function $w_{1,p}$ where $w_{1,p}(h)=1$ and $w_{1,p}(s)=p$.  Of course $f',g'\in C$, and we write $\widetilde{f'}$, $\widetilde{g'}$ as the leading parts of $f'$ and $g'$ respectively relative to $w_{1,p}$. To any polynomial $P(x,y)$, let $\overline{P}$ denote the leading part relative to the weighted degree function $w_{q,p}$.  Let $deg$ be the homogenous degree of $A_n$, and we have:

\begin{lemma}[\textbf{on\ degrees}]
$\overline{P}(\widetilde{f'},\widetilde{g'})\not=0$ and
$$\deg(P(f',g'))\geq\deg(\overline{P}(\widetilde{f'},\widetilde{g'})).$$
\end{lemma}

\begin{proof}





\noindent Consider $P(f',g')=Q(h,h^{-1},s)$ as well as $\overline{P}(\widetilde{f'},\widetilde{g'})=R(h,h^{-1},s)$ as the element of $C$, and then $R$ is the leading part of $Q$ relative to $w_{1,p}$, so all the monomials of $R$ appear in $Q$ with nonzero coefficients. Since $h$ and $v(s)$ are algebraically independent, $\deg(P(f',g'))=w_{\deg(h), \deg(s)}(Q(h,h^{-1},s))$ and $\deg(\overline{P}(\widetilde{f'},\widetilde{g'}))=w_{deg(h),deg(s)}(R(h,h^{-1},s))$. We conclude by the definition of weighted degree.
\end{proof}


\noindent Now we only need to estimate $\deg(\overline{P}(\widetilde{f'},\widetilde{g'}))$.

\

\noindent The following procedure is similar to the counterparts in \cite{MLY}.

\

\noindent Now we can write $\widetilde{f}=t^m$ and $\widetilde{g}=t^n+s$ just for convenience since $\deg(f)=m$ and $\deg(g)=n$. Then $\deg(t)=1$ and to each polynomial $m(x,y)$, $deg(m(t,s))=\deg_{1,\deg(s)}(m(x,y))$, or we can say that $v(t)$ and $v(s)$ are algebraically independent over $K$.

\

\noindent Let $N=w_{m,n}(\overline{P}(x,y))$, and $q$ be the greatest integer among the integers which are not greater than $\frac{N}{m+n}$(or we can denote it by $q=[\frac{N}{m+n}]$).  Define $Q(t,s)=\overline{P}(t^m,t^n+s)$, and we have

\begin{lemma}[\textbf{on\ monomials}]

There is a monomial $u(t,s)$ in $supp(Q)$ such that $\deg_s(u)\leq q$.
\end{lemma}

\begin{proof}
Choose a monomial $z(x,y)$ in $\text{supp}(\overline{P}(x,y))$ such that

\noindent (1) $\deg_y(z)$ is the greatest;

\noindent (2) Among all the monomials whose degree related to $y$ is equal to $\deg_y(z)$, $z$ is the greatest under the lexicographic order $x>>y$.

\

\noindent Let $z(x,y)=x^{\alpha_1}y^{\beta_1}\cdots x^{\alpha_k}y^{\beta_k}$ with $\alpha_1,\beta_k\geq 0$ and $\alpha_i\geq 1$, $2\leq i\leq k$, $\beta_j\geq 1$, $1\leq j\leq k-1$. Let $I=\deg_x(z)$ and $J=\deg_y(z)$. If $J\leq q$, then the degrees related to $s$ of all the monomials in $\text{supp}(Q)$ are not greater than $q$, and since in Proposition 2.4
 it is proved that $\text{supp}(Q)$ is not empty, we prove the lemma. Hence assume $J>q$, or $J\geq q+1$. If $I+J\geq 2q+2$, then since $N=mI+nJ$, $N=m(I+J)+(n-m)J\geq m(2q+2)+(n-m)(q+1)=(m+n)(q+1)$ which contradicts to $\frac{N}{m+n}<q+1$, and hence $I+J\leq 2q+1$.

\

\noindent Now for $z(x,y)=x^{\alpha_1}y^{\beta_1}\cdots x^{\alpha_k}y^{\beta_k}$, replace $x$ by $t^m$, and, if $\beta_i=2\sigma_i$, replace $y^{\beta_i}$ by $(st^n)^{\sigma_i}$; if $\beta_i=2\sigma_i+1$, replace $y^{\beta_i}$ by $(st^n)^{\sigma_i}s$. Then we get a monomial $u(t,s)$. It is easy to verify that $u(t,s)$ is a monomial in the extension of $z(t^m,t^n+s)=t^{m\alpha_1}(t^n+s)^{\beta_1}\cdots t^{m\alpha_k}(t^n+s)^{\beta_k}$, and the coefficient of $u$ is just the coefficient of $z$ in $\text{supp}(\overline{P})$ and hence nonzero.

\

\noindent Now we are going to prove that $u(t,s)$ cannot come from other extensions of the monomials in $\text{supp}(\overline P)$ after replacement.

\

\noindent We divide $z(x,y)$ into different parts first: $x^{\alpha_1}$; $y^{\beta_i}x^{\alpha_{i+1}}$ where $1\leq i\leq k-1$; $y^{\beta_k}$. Let $l(x,y)$ be a part of $z(x,y)$, and we define $\psi(l(x,y))$ being the corresponding part in $u(s,t)$ after replacement. So $\psi(x^{\alpha_1})=t^{m\alpha_1}$ and so on. If $u(s,t)$ is also in the extension of $z_1(t^m, t^n+s)$ where $z_1(x,y)\in supp(\overline{P}(x,y))$, then let $l_1(x,y)$ be a part of $z_1(x,y)$, and we define $\psi_1(l(x,y))$ to be the corresponding part in $u(s,t)$. Hence $z_1(x,y)$ can also be divided in to $\prod_{i+1^{k+1}}h_i(x,y)$ with $\psi_1(h_1)=\psi(x^{\alpha_1})$, $\psi_1(h_i)=\psi(y^{\beta_i}x^{\alpha_{i+1}})$ where $1\leq i\leq k-1$, and $\psi_1(h_{k+1})=\psi(y^{\beta_k})$. Obviously $deg_y(h_1)\geq deg_y(x^{\alpha_1})$. To each $i$, $1\leq i\leq k-1$, if $\beta_i$ is odd, then $\psi_1(h_{i+1})=(st^n)^{\sigma_i}s\cdot t^{m\alpha_{i+1}}$, and since $n<m$, the $t^n$ between two $s$ has to come from $\widetilde{g}$, so $h_{i+1}=y^{\beta_i}\cdot h'_{i+1)}$ where $\psi_1(h'_{i+1})=t^{m\alpha_{i+1}}$, namely $deg_y(h_{i+1})\geq\beta_i$. If $\beta_i$ is even, then $\psi_{h_{i+1}}=(st^n)^{\sigma}\cdot t^{m\alpha_{i+1}}=(st^n)^{\sigma-1}s\cdot t^{m\alpha_{i+1}+n}$. Hence $h_{i+1}=y^{\beta_i-1}h'_{i+1}$ where $\psi_1(h'_{i+1})=t^{m\alpha_{i+1}+n}$. However, since $n<m$, $h'_{i+1}$ cannot be of the form $x^p$ for some integer $p$, and hence $\deg_y(h'_{i+1})\geq 1$, namely $\deg_y(h_{i+1})\geq \beta_i$. To $h_{k+1}$, since $\psi_1(h_{k+1})=\psi_{y^{\beta_k}}=st^nst^n\cdots s$ or $st^nst^n\cdots st^n$, it has to equal to $y^{\beta_k}$.
Hence, $\deg_y(z_1(x,y))=\sum_{i=1}^{k+1}deg_y(h_i)\geq\sum_{i=1}^k\beta_i=deg_y(z(x,y))$. However $\deg_y(z(x,y))$ is the greatest one among the monomials in $supp((\overline P))$, $deg_y(z_1(x,y))=deg_y(z(x,y))$, and the only case is that $h_1(x,y)=x^{\alpha_1}$, and for $1\leq i\leq k-1$, $h'_{i+1}=x^{\alpha_{i+1}}$ if $\beta_i$ is odd and $\deg(h'_{i+1})=1$ if $\beta_i$ is even. Let $h'_{j+1}$ be the monomial with least $j$ such that $\beta_j$ is even but $h'_{j+1}\not=yx^{\alpha_{j+1}}$, then since $\deg_y(h'_{j+1})=1$, $h'_{j+1}=x^ryx^{\alpha_{j+1}-r}$ for $1\leq r\leq \alpha_{j+1}$. But if so, $z_1(x,y)>z(x,y)$ under the lexicographic order $x>>y$ which contradicts to $z(x,y)$ being maximal, hence no such $h'_{j+1}$ exists, namely each $h'_{j+1}$ of this kind is equal to $yx^{\alpha_{j+1}}$. Hence $z_1(x,y)=z(x,y)$ and the coefficient of $u(s,t)$ is not zero.

\

\noindent According to the definition of $u(s,t)$, we see that
$$\deg_s(u)\leq \sum_{i=1}^k\frac{\beta_i+1}{2}=\frac{J+k}{2}.$$ Obviously that $I\geq k-1$, and hence $\frac{J+k}{2}\leq\frac{I+J+1}{2}\leq\frac{2q+2}{2}=q+1$(be reminded that $I+J\leq q+1$). Notice that $\deg_s(u)=q+1$ only if all the $\beta_i$s' are odd and $I=k-1$, and $z(x,y)=y^{2\sigma_1+1}xy^{2\sigma_2+1}\cdots xy^{2\sigma_k+1}$ or $y^{2\sigma_1+1}xy^{2\sigma_2+1}\cdots xy^{2\sigma_{k-1}+1}x$. Then in $z(t^m,t^n+s)$ we replace $y^{2\sigma_1+1}x$ by $(t^ns)^{\sigma_1}t^n\cdot t^m$ and choose $u(t,s)=(t^ns)^{\sigma_1}t^nt^m(st^n)^{\sigma_2}s\cdots$. We denote $z(x,y)=y^{2\sigma_1+1}x\cdot h(x,y)$ and if $u(s,t)$ can also come from another monomial $z_1(x,y)$, then $z_1(x,y)=y^{2\sigma_2}h_1(x,y)h(x,y)$ with $\psi_1(h_1)=t^{m+n}$. Hence $h_1(x,y)=xy$ or $yx$. Notice again that $z(x,y)$ is the maximal element under the lexicographic order $x>>y$, and hence $h_1(x,y)=yx$ which means $z_1(x,y)=z(x,y)$. Then the coefficient of $u(s,t)$ is nonzero and $\deg_s(u)=q+1-1=q$.
\end{proof}

\

\noindent {\bf Proof\ of\ Theorem\ 1.1.}
Recall that $\deg(f)=m$, $\deg(g)=n$, $\deg(t)=1$, $\deg(s)=\deg([f,g])-\deg(f)=\deg([f,g])-m$, $N=w_{m,n}(\overline{P}(x,y))$. We have proved that there exists some

\noindent $u(s,t)\in\text{supp}(\overline{P}(t^m,t^n+s))$ such that $\deg_s(u)\leq N/(m+n)$. Since $N=deg_t(u)+n\cdot\deg_s(u)$, then $\deg(u)=\deg_t(u(t,s))+\deg_s(u(t,s))\cdot(\deg([f,g])-m)=N+\deg_s(u(s,t))(\deg([f,g])-m-n)$.
\noindent Since $\deg([f,g])-m-n\leq 0$, we get

$$\deg(P(f,g))\geq \deg(\overline{P}(\widetilde{f},\widetilde{g}))\geq \deg(u)\geq N+\frac{N(\deg([f,g])-m-n)}{m+n}$$
$$=\frac{\deg([f,g])}{m+n}w_{m,n}(P).$$
Since $m+n=\deg(fg)$, we get
$$\deg(P(f,g))\geq \frac{\deg([f,g])}{\deg(fg)}w_{\deg(f), \deg(g)}(P).$$\qed

\begin{example}
Let $f=x^n$, $g=x^m+y$, $P=[x,y]^k$. Then
$$\deg(P(f,g))=k(n+1)=\frac{\deg([f,g])}{\deg(fg)}w_{\deg(f), \deg(g)}(P),$$
which shows the estimate is sharp.
\end{example}

\begin{remark}
The methodology in this paper, unlike that in \cite{MLY}, is not applicable for commutative case, as in that case there is no invariant  to judge whether two polynomials are algebraically dependent or independent over a field of positive characteristic, and in fact to find such an invariant is an interesting question,
and it is also interesting to get a sharp degree estimate for the commutative case for positive characteristic.
\end{remark}

\

\section{\bf Acknowledgements}

\noindent Jie-Tai Yu is grateful to Academia Sinica Taipei, Shanghai University, Tata Institute of Fundamental Research, and Osaka University for warm hospitality and stimulating atmosphere during his visit, when part of this work was done. The authors would like to thank I-Chiau Huang for providing \cite{B1, B2} and helpful discussion,
and to thank Alexei Belov, Vesselin Drensky and Leonid Makar-Limanov for useful comments and suggestions.

\

\end{document}